\documentclass{amsart}

\usepackage{mathtools}
\usepackage{xypic}
\usepackage{amsthm}
\usepackage{amsmath}
\usepackage{amssymb}
\usepackage{mathrsfs}
\usepackage{cite}
\usepackage{graphicx}

\newtheorem{thm}{Theorem}
\newtheorem{question}[thm]{Question}
\newtheorem{lemma}[thm]{Lemma}
\newtheorem{prop}[thm]{Proposition}
\newtheorem{cor}[thm]{Corollary}
\newtheorem{defn}[thm]{Definition}
\newtheorem{remark}[thm]{Remark}
\newtheorem{conj}[thm]{Conjecture}

\newcommand{\nc}{\newcommand}
\nc{\rnc}{\renewcommand}
\rnc{\P}{\mathbb P}
\nc{\R}{\mathbb R}
\nc{\C}{\mathbb C}
\nc{\A}{\mathbb A}
\nc{\Q}{\mathbb Q}
\nc{\Z}{\mathbb Z}
\rnc{\O}{\mathcal O}
\nc{\LL}{\mathbb L}
\nc{\Hom}{\text{Hom}}
\nc{\codim}{\text{codim}}
\nc{\Sym}{\text{Sym}}
\nc{\Spec}{\text{Spec}\,}
\nc{\End}{\text{End}}
\nc{\eps}{\epsilon}
\nc{\Pic}{\text{Pic}}
\nc{\ov}{\overline}
\nc{\X}{\mathcal X}
\nc{\F}{\mathbb F}
\nc{\G}{\mathbb G}
\nc{\E}{\mathcal E}
\rnc{\S}{\mathcal S}

\rnc{\L}{\mathcal L}
\rnc{\H}{\mathbb H}
\nc{\Bl}{\text{Bl}}
\nc{\im}{\text{im}}
\nc{\gr}{\text{gr}}

\nc{\us}{\underset}
\nc{\ul}{\underline}
\nc{\bs}{\backslash}
\nc{\os}{\overset}

\nc{\Mod}{\text{Mod}}
\nc{\QMod}{\text{QMod}}
\nc{\MW}{\text{MW}}
\nc{\NS}{\text{NS}}
\nc{\Res}{\text{Res}}
\nc{\Aut}{\text{Aut}}
\nc{\W}{\mathcal W}
\nc{\NL}{\text{NL}}
\nc{\mult}{\text{mult}}

\nc{\U}{\mathscr U}
\nc{\D}{\mathscr D}

\rnc{\v}{{\langle v \rangle}}
\nc{\vdim}{\text{vdim}}

\title{Quasi-modular forms from mixed Noether-Lefschetz theory}

\author{Fran\c{c}ois Greer}
\begin{document}

\begin{abstract}
The Gromov-Witten theory of threefolds admitting a smooth K3 fibration can be solved in terms of the Noether-Lefschetz intersection numbers of the fibration and the reduced invariants of a K3 surface.  Toward a generalization of this result to families with singular fibers, we introduce completed Noether-Lefschetz numbers using toroidal compactifications of the period space of elliptic K3 surfaces.  As an application, we prove quasi-modularity for some genus 0 partition functions of Weierstrass fibrations over ruled surfaces, and show that they satisfy a holomorphic anomaly equation.
\end{abstract}

%\classification{14Nxx}
%\keywords{Gromov-Witten invariant, modular form, Hodge theory}

\maketitle

\section{Introduction}\noindent
The Gromov-Witten (GW) invariants of a smooth projective variety $X$ over $\C$ are virtual counts of curves on $X$ in each genus $g$ and homology class $\alpha\in H_2(X,\Z)$.  For the counts to be well-defined (without insertions), we require that the virtual dimension of the Kontsevich moduli space to be zero:
\begin{equation}\label{vdim}\vdim\, \ov{M}_g(X,\alpha) = c_1(X)\cdot \alpha + (1-g) (\dim X -3)=0.\end{equation}
It is convenient to assemble the GW invariants of $X$ into a potential function.  At fixed genus $g\geq 0$, we set
$$N^X_{g,\alpha} = \deg \left[ \ov{M}_g(X,\alpha) \right]^{vir},$$
\begin{equation}
F_g^X(q) = \sum_{\alpha\in H_2(X,\Z)} N^X_{g,\alpha}\,q^\alpha .\label{potential}\end{equation}
A striking phenomenon is the dependence of $F^X_g(q)$ on a finite amount of data; the series lies in some finite-dimensional vector space dictated by the discrete invariants of $X$.  For example, when $X$ is a quintic threefold, $F^X_g(q)$ lies in the finitely generated graded algebra of Yau-Yamaguchi.  In this paper, we study the case where $X$ is elliptic, and $\dim(X)\geq 3$.
\begin{defn}
A smooth projective variety $X$ is elliptic if it admits a flat proper surjection $\pi$ onto a smooth variety $B$, with generic fiber a smooth curve of genus one, and a regular section $z:B\to X$.
\end{defn}\noindent
The existence of the section endows the smooth fibers of $\pi$ with an elliptic curve group law.  It also gives a splitting of the short exact sequence
$$0\to \ker(\pi_*) \to H_2(X,\Z) \to H_2(B,\Z)\to 0.$$
The following modularity conjecture describes the potential functions of (\ref{potential}) in terms of the fibration structure $\pi$.
\begin{conj}\label{folkconj}
If $\beta\in H_2(B,\Z)$ is a primitive base class, then the relative potential function
$$F^X_{g,\beta}(q) = \sum_{\pi_*\alpha = \beta} N^X_{g, \alpha} q^{\alpha\cdot z_*[B] +\frac{k}{2}}  = \varphi(q)\cdot \Delta(q)^{-k},$$
where $k=\beta\cdot c_1(\pi_*\omega_{X/B})$, $\varphi(q)$ is a quasi-modular form for $SL_2(\Z)$, and
$$\Delta(q) = q\prod_{m\geq 1} (1-q^m)^{24}.$$
\end{conj}\noindent
This conjecture originated in the physics literature \cite{kmw}.  See \cite{obpix} for a more general conjecture at the level of cycles with insertions.\\\\
In Section \ref{program}, we present a program for proving Conjecture \ref{folkconj} in general using Noether-Lefschetz cycles.  This paper represents the first step in that program for genus 0.  We describe modular completions of period maps for families of elliptic K3 surfaces with Type II degenerations.  By performing intersections with completed Noether-Lefschetz divisors on a toroidal compactification, we find that the non-holomorphic behavior is explained by the boundary contribution.
\begin{defn}
A ruled surface $B$ is a smooth projective surface (other than $\P^2$) birational to $M\times \P^1$, with $M$ a smooth curve.
\end{defn}
\begin{prop}
Any ruled surface $B$ is isomorphic to an iterated blow up of a $\P^1$-bundle $\P E \to M$ ($b$ times).
\end{prop}
\begin{proof}
See for instance Ch. III of \cite{beauville}.
\end{proof}\noindent
Our main theorem concerns Weierstrass elliptic fibrations over ruled surfaces:
\begin{thm}\label{main}
Let $B$ be a ruled surface, and $X$ a smooth Weierstrass model over $B$ with fundamental line bundle $\pi_*\omega_{X/B} \simeq \omega_B^{-1}\otimes L_M$, where $L_M$ is the pullback of a line bundle on $M$.  If $\ell\in H_2(B,\Z)$ is the ruling class, and $f\in H_2(X,\Z)$ is the elliptic fiber class, then 
$$F^X_{0,\ell}(q) = \sum_{n\geq 0} N^X_{0,\ell+nf} q^n = \varphi(q)\cdot \frac{q}{\Delta(q)}.$$
Furthermore, $\varphi(q)\in \Q[E_2,E_4,E_6]_{10}$ satisfies a holomorphic anomaly equation:
$$\frac{\partial \varphi}{\partial E_2} = -\frac{b}{12}\cdot E_8,$$
where $b$ is the number of broken fibers of the ruled surface $B$.
\end{thm}
\noindent
{\bf Acknowledgments.}  The author is grateful to Adrian Brunyate, Philip Engel, Radu Laza, Davesh Maulik, and Georg Oberdieck for illuminating conversations.

\section{Sketch of a program} \label{program}\noindent
An old approach to computing Gromov-Witten invariants on $X$ is to embed
$$\iota: X\hookrightarrow P$$
as a complete intersection in a more positive variety $P$, often with a torus action.  Composing with $\iota$ induces an embedding of Kontsevich spaces
$$\ov{M}_g(X,\beta) \hookrightarrow \ov{M}_g(P,\iota_*\beta).$$
When the latter is smooth, the invariant $N^X_{g,\beta}$ can be realized as the integral of a top class on $\ov{M}_g(P,\iota_*\beta)$.\\\\
Here we explore a dual construction; suppose instead that we are given a proper surjection
$$\pi: X \twoheadrightarrow B$$
onto a more positive variety $B$ with $\dim(X)>\dim B\geq 2$.  For 
$$\alpha\notin \ker(\pi_*:H_2(X)\to H_2(B)),$$
composing with $\pi$ induces a morphism of Kontsevich spaces
$$\ov{M}_g(X,\alpha) \to \ov{M}_g(B,\pi_*\alpha).$$
The fibers of this morphism are Kontsevich spaces with smaller dimensional target, and the image of this morphism is a proper subvariety of $\ov{M}_g(B,\pi_*\alpha)$ whose class can be computed using topological methods.  
\begin{prop}\label{ktrivial}
If the fibers of $\pi$ are Gorenstein and $K$-trivial, then 
$$\text{vdim}\,\ov{M}_g(X,\alpha) = \text{vdim}\,\ov{M}_g(X,\alpha')$$
for $\alpha - \alpha' \in \ker(\pi_*)$.
\end{prop}
\begin{proof}
By the $K$-triviality condition, $\omega_{X/B}$ is the pullback of a line bundle on $B$, so $\omega_X \simeq \pi^*\omega_B \otimes \omega_{X/B}$ is also a pullback.  By the projection formula, $ K_X\cdot (\alpha-\alpha') = 0$.  The statement now follows from the virtual dimension formula (\ref{vdim}).
\end{proof}\noindent
Proposition \ref{ktrivial} implies that the virtual dimension is unchanged by the addition of fiber classes.  This leads us to define generating series of Gromov-Witten invariants by summing in the fiber direction, as in Conjecture \ref{folkconj}.  We hope to implement the strategy below for general $K$-trivial fibrations.  For now, we restrict to the case where the fibers are elliptic curves, but $B$ can have any dimension $\geq 2$.   \\\\
Let $\pi:X\to B$ be an elliptic fibration, and let $Z = z(B)\subset X$ the image of the section.  The fundamental line bundle is defined by
$$L = (z^* N_{Z/X})^\vee.$$
The line bundle $\O_X(3Z)$ defines a birational morphism $X\to X^W$, where $X^W\subset \P(L^2\oplus L^3 \oplus \O_B)$ is cut out by a global Weierstrass equation
\begin{equation}\label{weier}
y^2z = x^3 + a_4 xz^2 + b_6z^3
\end{equation}
with $a_4\in H^0(L^4)$ and $b_6\in H^0(L^6)$.  The new variety $X^W$ is called the {\it Weierstrass model}.  If $X^W$ is smooth, then its GW invariants are well-defined and can be related to those of $X$ via blow up formulae \cite{blowup}.  If $X^W$ is singular, then one can consider a smooth deformation, whose GW invariants are related to those of $X$ via transition formulae \cite{aliruan}.
\begin{defn}
An elliptic fibration is called {\it Weierstrass} if it is isomorphic to its Weierstrass model.
\end{defn}
\begin{prop}
The fundamental line bundle $L$ of a Weierstrass fibration $\pi:X\to B$ is isomorphic to $\pi_*\omega_{X/B}$, and it is nef.
\end{prop}
\begin{proof}
For any curve $C\subset B$, the fibered product
$$X_C := X\times_B C$$
is a Weierstrass fibration over $C$ with fundamental line bundle $L|_C$.  Thus, $L^4|_C$ and $L^6|_C$ have sections, so $\deg(L|_C)\geq 0$.  For the first statement, apply the adjunction formula to the global Weierstrass equation (\ref{weier}).
\end{proof}\noindent
With these considerations in mind, we make the following simplifying assumption:\\\\
{\bf Assumption A.}  $\pi:X\to B$ {\it is a Weierstrass fibration whose fundamental line bundle $L$ is ample.  Furthermore, $L^4$ and $L^6$ are very ample.}\\\\
Using the virtual dimension formula (\ref{vdim}), we observe that,
\begin{align*}
0=\vdim \, \ov{M}_g(X, \alpha) &= -K_X\cdot \alpha +(1-g)( \dim X -3) \\
 &= -(K_B+c_1(L))\cdot \beta + (1-g)(\dim B -2)\\
 &= \vdim\, \ov{M}_g(B,\beta) - c_1(L)\cdot\beta +(1-g).
\end{align*}
Setting $k=c_1(L)\cdot \beta>0$, the equation above becomes
\begin{equation}\label{dim}
\vdim\, \ov{M}_g(B,\beta) = k +g-1.
\end{equation}
{\bf Assumption B.} {\it The general element $[C\to B]\in \ov{M}_g(B,\beta)$ has smooth domain.}\\\\
This can be thought of as a positivity assumption, and as a restriction to genus $g=0$.  Assumption B can be removed, but we must treat the other components of $\ov{M}_g(B,\beta)$ separately in the period map construction below.\\\\
Let $p:\mathcal C \to \ov{M}_g(B,\beta)$ be the universal curve, and $u:\mathcal C \to B$ the universal map.  Consider the family of elliptic surfaces $q:\mathscr S \to \ov{M}_g(B,\beta)$ given by
$$\mathscr S := X \times_B \mathcal C \to \ov{M}_g(B,\beta)$$
The fiber of $q$ at a stable map $[C\to B]\in \ov{M}_g(B,\beta)$ is the elliptic surface
$$X_C = X\times_B C.$$
Assumption B makes $q$ is generically smooth, so we have a rational period map
$$\rho: \ov{M}_g(B,\beta) \dasharrow \Gamma \bs \mathscr D$$
to the appropriate period space of polarized Hodge structures for elliptic surfaces.  The map $\rho$ is only defined for $[C\to B]$ with smooth domain; otherwise the surface $S=X_C$ would have a normal crossing singularity.  Each elliptic surface $S$ in the family is Weierstrass with fundamental line bundle $L|_C$ of degree $k$, and thus has
\begin{equation}\label{codim} p_g(S) = k+g-1.\end{equation}
by a standard calculation \cite{mirandabasic}.  This matches with the virtual dimension in (\ref{dim}).  Inside the period space $\Gamma\bs \mathscr D$, we have a countable collection of codimension $p_g(S)$ Noether-Lefschetz (NL) cycles\footnote{See Section \ref{hodge} for a formal definition of the components $\NL_n$ of the Noether-Lefschetz locus.}, which parametrize Hodge structures with extra algebraic curve classes.  Set theoretically, the intersection
$$\rho\left( \ov{M}_g(B,\beta) \right) \cap \NL$$
contains the image of $\ov{M}_g(X,\alpha)$ with smooth domain curves.  The primary challenge is to make this intersection topological.  The issue is that both the period space $\Gamma\bs \mathscr D$ and the Noether-Lefschetz cycles are non-compact.
\begin{conj}\label{pipedream}
There exists a smooth completion $(\Gamma\bs\mathscr D)^*$ such that $\rho$ extends to a morphism
$$\ov{\rho}: \ov{M}_g(B,\beta) \to (\Gamma\bs \mathscr D)^*.$$
Furthermore, the series
$$\varphi(q) = \sum_{n\geq 0} \left(\ov{\rho}_*[\ov{M}_g(B,\beta)]^{vir}\cap\ov{\NL}_n\right) q^n$$
is quasi-modular, and it agrees with the form $\varphi(q)$ in Conjecture \ref{folkconj}.
\end{conj}\noindent
Our Theorem \ref{main} confirms Conjecture \ref{pipedream} in the case where $g=0$, $k=2$, and $\ov{M}_0(B,\beta)$ is smooth of the expected dimension.  Indeed, consider the Cartesian square
\begin{equation} \label{cartesian}
\xymatrix{
\mathscr S\ar[r]^{\widetilde{u}} \ar[d]_{\widetilde{\pi}} & X \ar[d]_\pi \\
\mathcal C \ar[r]^u \ar[d]_p & B \\
\ov{M}_0(B,\beta) .&
}
\end{equation}
Assumption A implies that $\mathscr S \to \mathcal C$ is a smooth Weierstrass fibration with fundamental line bundle $u^*L$, if $X\to B$ is chosen generically.  Since $g=0$, the surface $\mathcal C$ is ruled:
\begin{thm} \cite{beauville}
(Tsen) Let $\mathcal C$ be a projective surface with a morphism $\mathcal C\to M$ to a smooth curve.  If the generic fiber is geometrically integral of genus 0, then $\mathcal C$ is birational to $M\times \P^1$.
\end{thm}
\noindent
We have $\ov{M}_0(X,\beta+nf) \simeq \ov{M}_0(\mathscr S, \ell+nf)$, since any stable map to $X$ factors through $\mathscr S$, and in fact the virtual classes agree too.
\begin{prop}
$\left[ \ov{M}_0(X,\beta+nf) \right]^{vir} = \left[ \ov{M}_0(\mathscr S,\ell+nf) \right]^{vir}$
\end{prop}
\begin{proof}
The relative obstruction theory used to define the virtual class on $\ov{M}_0(X)$ is $Rp_{X*}u_X^* T_X$, where $p_X$ is the universal curve, and $u_X$ is the universal map.  The short exact sequence
$$0 \to T_{X/B} \to T_X \to \pi^* T_B\to 0$$
induces an exact triangle in the derived category of $\ov{M}_0(X)$:
$$Rp_{X*}u_X^* T_{X/B} \to Rp_{X*}u_X^* T_X \to Rp_{X*}u_X^* \pi^*T_B.$$
The analogous construction for $\mathscr S \to \mathcal C$ yields an exact triangle in the derived category of $\ov{M}_0(\mathscr S)$:
$$Rp_{\mathscr S*}u_{\mathscr S}^* T_{\mathscr S/\mathcal C} \to Rp_{\mathscr S*}u_{\mathscr S}^* T_{\mathscr S} \to Rp_{\mathscr S*}u_{\mathscr S}^* \widetilde{\pi}^*T_{\mathcal C}.$$
Since the square in (\ref{cartesian}) is Cartesian, we have $T_{\mathscr S/\mathcal C} \simeq \widetilde{u}^* T_{X/B}$, which induces an isomorphism 
$$Rp_{\mathscr S*}u_{\mathscr S}^* T_{\mathscr S/\mathcal C} \simeq Rp_{X*}u_X^* T_{X/B}.$$
The derivative map $T_{\mathcal C} \to u^* T_B$ gives a map $\widetilde{\pi}^* T_{\mathcal C} \to  \widetilde{u}^*\pi^* T_B $.  This induces a quasi-isomorphism
$$Rp_{\mathscr S*}u_{\mathscr S}^* \widetilde{\pi}^*T_{\mathcal C} \simeq Rp_{X*}u_X^* \pi^*T_B,$$
since the horizontal component of each stable map is unobstructed in $B$.  By the completion axiom of a triangulated category, we have an isomorphism between the middle terms of the exact triangles.
\end{proof}
\noindent
An advantage of the $k=2$ case is that the generic fiber of the family 
$$q:\mathscr S \to \ov{M}_0(B,\beta)$$
is an elliptic K3 surface.  Period spaces for K3 surfaces have smooth compactifications, which are used to produce the quasi-modularity result.  If $k=1$, the generic fiber is a rational elliptic surface, which has trivial periods, and its Gromov-Witten theory is described in \cite{obpix}.  If $k\geq 3$, the problem of finding a smooth completion of the period space is still open.  See \cite{fgreer} for some modularity results for $k\geq 3$ when all the domain curves in $\ov{M}_0(B,\beta)$ are smooth.
\bigskip

\section{Reductions}\label{reduction}\noindent
In this section, we reduce Theorem \ref{main} to the case where 
\begin{itemize}
\item $L_M$ is sufficiently positive, and
\item the birational morphism $B \to \P E$ is a blow up of $b$ points in distinct fibers of the $\P^1$ bundle $\P E \to M$.
\end{itemize}
We employ the degeneration formula of \cite{jli} to relate the invariants in Theorem \ref{main} to invariants of a more flexible geometry.  The strategy is to construct an elliptic fibration $\mathcal X \to \mathcal B\to \Spec \C[[t]]$ which realizes a normal crossings degeneration
$$\mathcal X_t \rightsquigarrow \mathcal X_0,$$
where $\mathcal X_0$ has star-shaped dual graph, and the gluing loci are all K3 surfaces.  The general fiber $\mathcal X_t$ and the extremal components of the central fiber $\mathcal X_0$ will all satisfy the reduction conditions.
\begin{lemma}\label{nefbig}
Let $B$ be as in Theorem \ref{main}.  If $\deg(L_M) \gg 0$, then $\omega_B^{-1}\otimes L_M$ is nef and big.
\end{lemma}
\begin{proof}
The line bundle $\omega_B^{-1}\otimes L_M$ has class $2\zeta + l\ell - \sum e_i$, and we assume that $l\gg 0$.  Any irreducible vertical curve on $B$ has class $\ell$, $\ell-e_i$, $e_i$, or $e_i-e_j$.  The degree of $\omega_B^{-1}\otimes L_M$ is nonnegative on each of these.  Curves without vertical components will have classes of the form
$$m\zeta + d\ell - \sum m_i e_i,$$
where each $m_i\leq m$, and $d\geq s m$ for some slope $s\in \Q$ depending only on $\P E$.  Now, 
\begin{align*}
\left(m\zeta+d\ell - \sum_{i=1}^b m_i e_i \right)\cdot \left(2\zeta + l\ell - \sum_{i=1}^b e_i \right) &= 2m\zeta^2 + 2d + ml - \sum m_i\\
 &\geq 2m\zeta^2 + 2sm + ml - mb\\
 &= m(2\zeta^2 + 2s + l - b)\\
 &\geq 0,
\end{align*}
since $\zeta^2$, $s$, and $b$ are constants independent of the test curve.  To see that $\omega_B^{-1}\otimes L_M$ is big, one checks that its self-intersection is positive.
\end{proof}\noindent
Lemma \ref{nefbig} will allow us to lift Weierstrass equations, using the Kawamata-Viehweg vanishing theorem, which we state here for reference.
\begin{thm} \label{kawvie} (Kawamata-Viehweg)
If $L$ is a nef and big line bundle on a smooth projective variety $X$, then for $i>0$,
$$H^i(X,\omega_X\otimes L)=0.$$
\end{thm}\noindent
We are now ready to construct the degeneration.  Start with $X\to B$ a Weierstrass fibration with fundamental line bundle $\omega_B^{-1} \otimes L_M$, as in Theorem \ref{main}, with no additional assumption on $L_M$.
\begin{cor}\label{van1}
There exists some $m\gg 0$ such that for any divisor $D$ on $M$ of degree $m$, we have
\begin{align*}
H^1(B,\omega_B^{-4}\otimes L_M^4 \otimes \O_M(D))&=0;\\
H^1(B,\omega_B^{-6}\otimes L_M^6 \otimes \O_M(D))&=0.
\end{align*}
\end{cor}
\begin{proof}
This follows easily from Lemma \ref{nefbig} and Theorem \ref{kawvie}.
\end{proof}\noindent
{\bf Step 1.}  Construct a stable curve $M^{ct}$ of compact type by gluing $M$ to $m$ different smooth curves $M^{(i)}$ at general points $p_1,p_2,\dots,p_m\in M$.\\\\
{\bf Step 2.}  Recall that $E\to M$ is a rank 2 vector bundle, and $\P E$ its projectivization.  Choose rank 2 vector bundles $E^{(i)}$ on each new component, and identify their fibers over the nodes to obtain a bundle $E^{ct}\to M^{ct}$.\\\\
{\bf Step 3.}  Let $\mathcal M \to \Spec \C[[t]]$ be a smoothing deformation of $M^{ct}$, and let $\mathcal E\to \mathcal M$ be a rank 2 bundle extending $E^{ct}$.  Taking the projectivization
$$\P \mathcal E \to \mathcal M,$$
we obtain a degeneration 
$$\P E_t \rightsquigarrow \P E^{ct} = \P E\cup_{\P^1} \P E^{(1)} \cup_{\P^1} \dots \cup_{\P^1} \P E^{(m)}.$$
{\bf Step 4.}  Deform the blow up centers.  Inductively, suppose we have a degeneration $\mathcal B \to \Spec \C[[t]]$ with central fiber 
$$B\cup_{\P^1} \P E^{(1)} \cup_{\P^1} \dots \cup_{\P^1} \P E^{(m)},$$
and let $y\in B$ be a point away from the glued $\P^1$.  Choose a general section $\mathcal Y$ of $\mathcal B \to \Spec\C[[t]]$ specializing to $y$.  The blow up $\Bl_{\mathcal Y} \mathcal B$ specializes to 
$$\Bl_y B \cup_{\P^1}\P E^{(1)} \cup_{\P^1} \dots \cup_{\P^1} \P E^{(m)}.$$
Hence, we can construct a degeneration whose generic fiber is $\P E_t$ blown up at $b$ points in distinct fibers of the projective bundle $\P E_t \to \mathcal M_t$.\\\\
{\bf Step 5.}  Since $\mathcal M$ is a compact type degeneration, it has a proper relative Jacobian.  Let $\mathcal L_{\mathcal M}$ be a line bundle on $\mathcal M$ such that $\mathcal L_{\mathcal M}|_{M} = L_M$, and
\begin{align*}
\deg(\mathcal L_{\mathcal M}|_{\mathcal M_t} ) & \gg 0; \\
\deg(\mathcal L_{\mathcal M}|_{M^{(i)}} ) & \gg 0. \\
\end{align*}
{\bf Step 6.}  We construct a Weierstrass fibration $\mathcal X \to \mathcal B$ with fundamental line bundle
$$\mathcal L := \omega_{\mathcal B}^{-1} \otimes \O_{\mathcal B}(-B) \otimes \mathcal L_{\mathcal M},$$
which restricts to $\omega_B^{-1}\otimes L_M$ on $B$ by adjunction.  By the Leray spectral sequence,
\begin{align*}
H^1(\mathcal B, \mathcal L^4 \otimes \O_{\mathcal B}(-B))  &= H^1(B,\omega_B^{-4}\otimes L_M^4 \otimes \O_B(-B)); \\
H^1(\mathcal B, \mathcal L^6 \otimes \O_{\mathcal B}(-B))  &= H^1(B,\omega_B^{-6}\otimes L_M^6 \otimes \O_B(-B)).
\end{align*}
Both of these vanish by Corollary \ref{van1}.  Hence, we get surjections of Weierstrass coefficient spaces:
\begin{align*}
H^0(\mathcal B, \mathcal L^4 ) &\twoheadrightarrow H^0(B,\omega_B^{-4}\otimes L_M^4); \\
H^0(\mathcal B, \mathcal L^6 ) &\twoheadrightarrow H^0(B,\omega_B^{-6}\otimes L_M^6).
\end{align*}
This allows us to choose $\mathcal X\to \mathcal B$ extending $X\to B$.  Since $\mathcal L$ restricts to $\O_{\P^1}(2)$ on each gluing locus in the central fiber, we have
$$\mathcal X_0 = X \cup_{K3} X^{(1)} \cup_{K3} \dots \cup_{K3} X^{(m)},$$
the desired star-shaped degeneration.  Note that $X^{(i)} \to \P E^{(i)}$ satisfies the second reduction condition vacuously, since the projective bundle is not blown up.\\\\
Since the GW theory of a K3 surface is trivial, the degeneration formula of \cite{jli} implies that
$$F^{\mathcal X_t}_{0,\ell}(q) = F^X_{0,\ell}(q)+\sum_{i=1}^m F^{X^{(i)}}_{0,\ell}(q).$$
The properties stated in Theorem \ref{main} are linear, so it suffices to prove them for $\mathcal X_t$ and for the extremal components $X^{(i)}$, all of which satisfy the reduction conditions.
\bigskip

\section{Hodge Theory}\label{hodge}
\noindent
Recall that a (smooth) K3 surface $S$ has middle cohomology lattice
\begin{equation}\label{marking}
H^2(S,\Z) \simeq II_{3,19}.
\end{equation}
The K3 surfaces that arise in this paper are elliptic; their Neron-Severi group contains a zero section class $z$ and a fiber class $f$, which span a sublattice:
$$U \simeq \begin{pmatrix} -2 & 1 \\ 1 & 0 \end{pmatrix}.$$
Any primitive embedding $U\subset II_{3,19}$ has orthogonal complement $\Lambda\simeq II_{2,18}$.  Via the isomorphism (\ref{marking}), a holomorphic $2$-form $\Omega$ on $S$ will lie in
$$\Lambda\otimes \C\subset H^2(S,\C).$$
This allows us to define a period domain for $U$-polarized K3 surfaces.  The following general definition is standard:
\begin{defn}
Let $\Lambda$ be an even unimodular lattice of signature $(2,l)$.  We set
$$\mathscr D(\Lambda) := \{\omega\in \P(\Lambda\otimes \C): (\omega,\omega)=0,\, (\omega,\ov{\omega})>0 \}^+.$$
\end{defn}
\noindent
It is well known that $\mathscr D(\Lambda)$ is a Hermitian symmetric domain of Type IV and complex dimension $l$.  To remove the ambiguity in the marking isomorphism (\ref{marking}), we quotient $\mathscr D(\Lambda)$ by the automorphism group $\Gamma$ of $\Lambda$ to obtain the global period space $\Gamma \bs\mathscr D(\Lambda)$, which is the analytification of a quasi-projective variety \cite{bb}.  The tautological line bundle $\O(-1)$ on $\P(\Lambda\otimes \C)$ descends to $\Gamma\bs\mathscr D(\Lambda)$, and there it is called the Hodge line bundle.
\begin{defn}
For a positive integer $n$, the Noether-Lefschetz locus is given by:
$$\NL_n :=\Gamma \bs \left(\bigcup_{\substack{v\in \Lambda \\ (v,v)=-2n}} v^\perp\right) \subset \Gamma\bs \mathscr D(\Lambda) .$$
A component of $\NL_n$ is called {\it primitive} if it is $\Gamma\bs v^\perp$ for $v$ a primitive vector.
\end{defn}\noindent
These are divisors in $\Gamma\bs\mathscr D(\Lambda)$ whose components are abstractly isomorphic to the quotient of $\mathscr D(v^\perp)$ by the stabilizer of $v$ in $\Gamma$.  They parametrize polarized Hodge structures of K3 type with an integral $(1,1)$-class $v$.\\\\
Both $\Gamma\bs\mathscr D(\Lambda)$ and $\NL_n$ are non-compact spaces.  As such, we can define fundamental cycle classes in the Borel-Moore homology:
$$[\NL_n] \in H^{BM}_{2l-2}(\Gamma\bs \mathscr D(\Lambda),\Q) \simeq H^2(\Gamma\bs \mathscr D (\Lambda),\Q).$$
The theta correspondence techniques of Borcherds and Kudla-Millson produce the following modularity property:
\begin{thm}\label{km}\cite{borch}\cite{km}
Let $\lambda\in H^2(\Gamma\bs \mathscr D(\Lambda),\Q)$ denote the Chern class of the Hodge line bundle.  The formal power series
$$\varphi(q):=-\lambda + \sum_{n\geq 1} [\NL_n]\, q^n$$
is an element of $\Mod(SL_2(\Z),rk(\Lambda)/2)\otimes H^2(\Gamma\bs \mathscr D(\Lambda),\Q)$.
\end{thm}\noindent
We will drop the reference to $\Lambda$ in the notation $\mathscr D(\Lambda)$ from now on, as we are only concerned with the case $\Lambda \simeq II_{2,18}$.  Since $\Mod(SL_2(\Z),10)$ has dimension 1, Theorem \ref{km} reduces to:
\begin{cor}\label{modularity}
For any homology class $\alpha\in H_2(\Gamma\bs\mathscr D,\Q)$,
$$\varphi_\alpha(q):= -\alpha\cdot \lambda + \sum_{n\geq 1} \alpha\cdot [\NL_n]\,q^n \in \Q\cdot E_{10}(q),$$
where $E_{10}(q)$ is the Eisenstein series of weight 10.
\end{cor}\noindent
In the context of Theorem \ref{main} (with reductions), the composition $X\to B\to M$ may be viewed as a flat family of elliptic K3 surfaces, with singular members.  The associated period map
\begin{equation}\label{periodmap}
\rho:M \dashrightarrow \Gamma\bs \mathscr D
\end{equation}
is indeterminate at the $b$ points of $M$ where the fiber of $B\to M$ is reducible.  These correspond to degenerations
\begin{equation}\label{degen}
S_t \rightsquigarrow S_0=R_1 \cup_{E} R_2
\end{equation}
of a K3 surface to a normal crossing of two rational elliptic surfaces obtained by identifying a pair of smooth elliptic fibers $E_1\subset R_1$ and $E_2\subset R_2$.  The limiting Hodge structure of this semistable degeneration is impure, so the rational period map (\ref{periodmap}) cannot be extended.  Indeed, next we will analyze the mixed Hodge structures associated to (\ref{degen}), working over $\Q$ for convenience.\\\\
The Mayer-Vietoris sequence applied to $S_0=R_1\cup_E R_2$ yields
\begin{equation}\label{mv}
0\to H^1(E) \to H^2(S_0) \to H^2(R_1)\oplus H^2(R_2) \to H^2(E) \to 0,
\end{equation}
so $H^2(S_0,\Q)= \Q^{21}$.  The Clemens-Schmid sequence of the degeneration yields
\begin{equation}\label{cs}
0 \to \Q\cdot [E_1-E_2] \to H^2(S_0) \to H^2_{lim}(S_t) \os{N}\to H^2_{lim}(S_t).
\end{equation}
Here, $N$ is the logarithm of the unipotent monodromy operator.  The image of $N$ is Poicar\'{e} dual to the vanishing cycles of the degeneration.  These can be described geometrically:  let $\gamma$ be a loop in the base $\P^1$ which gets pinched to a point $p$ in the degeneration of the base $\P^1 \rightsquigarrow \P^1 \cup_{pt} \P^1$.  The elliptic fibration $S_t \to \P^1$ is trivial over $\gamma$, so the homological vanishing cycles can be described as
$$\gamma \times H_1(E).$$
Now, the weight filtration $W_\bullet$ on $H^2_{lim}(S_t)$ is given by
\begin{align*}
0 \subset \im(N) \subset &\ker(N) \subset H^2(S_t)\\
0 \subset \Q^2 \subset &\Q^{20} \subset \Q^{22}.
\end{align*}
Using (\ref{mv}) and (\ref{cs}), the associated graded groups may be identified as follows:
\begin{align*}
\gr^W_1H^2_{lim}(S_t) &\simeq H^1(E)\\
\gr^W_2H^2_{lim}(S_t)&\simeq \{ (r_1,r_2\in H^2(R_1)\oplus H^2(R_2): r_1|_{E_1} = r_2|_{E_2} \}/\, \Q\cdot[E_1-E_2].
\end{align*}
These groups are endowed with pure Hodge structures:  $H^1(E)$ has the Jacobian structure, and $\gr^W_2$ has a weight 2 Hodge structure of Tate type.\\\\
The limiting mixed Hodge structure can be polarized by the sublattice $U = \langle z,f \rangle$.  First, the log monodromy operator $N$ is skew-symmetric with respect to the cup product, so we have
$$\ker(N) = \im(N)^\perp.$$
Since the zero section and fiber cycles extend over the central fiber of the degeneration (\ref{degen}), we have
$$\langle z, f \rangle \subset \ker(N)\subset H^2_{lim}(S_t),$$
by the Invariant Cycle Theorem.  This gives a polarized mixed Hodge structure on $\Lambda = U^\perp$ with weight filtration:
$$0 \subset \im(N) \subset \ker(N)\cap \Lambda \subset \Lambda.$$
\begin{remark}
The vanishing cycle group $\im(N)\subset \Lambda$ is a rank 2 isotropic sublattice, which plays the role of $J$ in Section \ref{compact}.  The quotient
$$\gr^W_2 \Lambda\simeq  \im(N)^\perp / \im(N)$$
can be identified with the lattice
$$H^2_{prim}(R_1) \oplus H^2_{prim}(R_2) \simeq (-E_8)\oplus (-E_8).$$
Here $H^2_{prim}(R_i)$ denotes the orthogonal complement of $\langle z_i,e_i\rangle$ in $H^2(R_i)$.
\end{remark}
\bigskip

\section{Compactifying the Period Space}\label{compact}
\noindent
In this section, we recall the different compactifications of the period space $\Gamma\bs\mathscr D$, following closely the exposition of \cite{looij}.  We can then extend the period map $\rho$ from (\ref{periodmap}) over the boundary, and take its intersection product with the closures of the Noether-Lefschetz divisors.\\\\
The Satake-Baily-Borel compactification is a projective variety with singularities at the boundary.  As a set, it can be described by adding to $\mathscr D$ a collection of boundary components corresponding to isotropic $\Q$-lines $I\subset \Lambda_\Q$ and isotropic $\Q$-planes $J\subset \Lambda_\Q$.  The lines $I$ correspond to points $p_I$ in the boundary, and the planes $J$ correspond to curves; let $\H_J$ denote the upper half-plane which occurs as the image of $\mathscr D$ under the linear projection $\pi_{J^{\perp}}:\P(\Lambda_\C) \dashrightarrow \P(\Lambda_\C/J^\perp_\C)$.  Taken together, these components admit an action of $\Gamma$ with quotient a locally compact Hausdorff space:
\begin{align*}
(\Gamma\bs \mathscr D)^{SBB}:&= \Gamma\bs \left(  \mathscr D \sqcup \bigsqcup_{\{I\}} \{p_I\} \sqcup\bigsqcup_{\{J\}} \H_J   \right)\\
 &= (\Gamma\bs \mathscr D) \sqcup \left(\bigsqcup_{\Gamma\bs\{I\}} \{p_I\}\right) \sqcup\left( \bigsqcup_{\Gamma\bs\{J\}} \Gamma_J\bs\H_J \right).
\end{align*}
Here, $\Gamma_J\subset \Gamma$ denotes the stabilizer of $J$ as a subspace of $\Lambda_\Q$, which acts on $\H_J$ through an arithmetic quotient group (see below).\\\\
As explained in \cite{brunyate}, the Satake-Baily-Borel compactification for our particular choice of $\Lambda$ is given by:
$$(\Gamma\bs \mathscr D)^{SBB} = (\Gamma\bs \mathscr D) \sqcup \left( \P^1 \cup_{p} \P^1 \right).$$
The $\P^1$ boundary components correspond to the two (up to $\Gamma$) isotropic planes $J\subset \Lambda_\Q$.  The quotient lattices $J^\perp/J$ are isomorphic to $(-E_8)^{\oplus 2}$ and $-D_{16}^+$, respectively.  Points in the boundary of $(\Gamma\bs\D)^{SBB}$ can be interpreted as associated graded pieces of limiting mixed Hodge structures.  For example, the first boundary component is identified with the $j$-line $\A^1_j= PSL(2,\Z)\bs \H$.  In the degeneration (\ref{degen}), the $j$-invariant of $E$ determines the local extension of $\rho$ to the boundary.\\\\
Mumford et al. \cite{amrt} have constructed smooth projective resolutions
$$\eps:(\Gamma\bs\mathscr D)^{\Sigma} \to (\Gamma\bs\mathscr D)^{SBB},$$
for a choice of fan decomposition $\Sigma$ of the nilpotent cone.  For our purpose, it suffices to describe the local geometry of this construction away from the cusps $p_I$.\\\\
Fix an isotropic $\Q$-plane $J\subset \Lambda_\Q$.  Let $G_J\subset O(\Lambda_\R)$ be its stabilizer in the indefinite orthogonal group.  By restriction, $G_J$ maps to $GL(J_\R)\times O(J^\perp/J)(\R)$, and the kernel $N_J$ is a real Heisenberg group.  These facts are summarized in the following diagram of real Lie groups:
$$
\xymatrix{
\,\,\,\,\,\,\,\,\,\,\,\,\,\,\,\,\,\,Z(N_J) \ar[d] \simeq \wedge^2 J_\R & & \\
N_J \ar@{->>}[d] \ar[r] & G_J \ar[r] & GL(J_\R) \times O(J^\perp/J)(\R)\\
(J^\perp /J)^\vee \otimes J_\R. & &
}
$$
Restricting to the discrete subgroup $\Gamma_J = G_J\cap \Gamma \subset O(\Lambda_\R)$, we obtain
$$
\xymatrix{
\Z \ar[d] & &  \\
\Gamma_N \ar@{->>}[d] \ar[r] & \Gamma_J \ar[r] & \ov{\Gamma}_J\subset GL(J) \times O(J^\perp/J)\\
L, & &
}
$$
where $L$ is a lattice in the real vector space $(J^\perp/J)^\vee\otimes J_\R$.  Consider the composition
$$\Gamma_N \bs \D \to \Gamma_N \bs\pi_J (\D) \to \pi_{J^\perp} (\D)=\H_J.$$
The first map is a punctured disk ($\Delta^*\simeq \Z\bs \H$) bundle, and the second map is principal real torus bundle whose structure group is the quotient of $(J^\perp/J)^\vee\otimes J_\R$ by $L$.  It inherits a complex structure such that the latter bundle is isogenous\footnote{In the case where $\Gamma$ is the full automorphism group of $\Lambda$, the isogeny is an isomorphism.} to a repeated fiber product of the tautological family of elliptic curves over $\H_J$.  If $\Gamma^J\subset \Gamma_J$ is the subgroup fixing $J$, we have the (orbifold) composition
$$\Gamma^J \bs \D \to \Gamma^J \bs\pi_J (\D) \to \pi_{J^\perp} (\D)=\H_J.$$
The fibers of the second map are now quotients of tori by the image of $\ov{\Gamma}_J$ in $O(J^\perp/J,\Z)$.  In the case where $\Gamma=O(\Lambda)$, the image is all of $O(J^\perp/J,\Z)$, so the quotients are weighted projective spaces by the following remarkable theorem of Looijenga:
\begin{thm}\cite{root}\label{root}
Let $R$ be a root system, $W(R)$ its Weyl group, and $Q$ its dual root lattice.  For any elliptic curve $E$,
$$W(R)\bs (E\otimes_\Z Q) \simeq \P(1,g_1,g_2,\dots,g_N),$$
with weights $g_i$ given by the coefficients of the highest coroot.  The hyperplane class pulls back to the sum of reflection hypertori in $E\otimes_\Z Q$.
\end{thm}
\noindent
In our case, the root system $R=E_8^{\oplus 2}$ is self-dual.  The local structure of the toroidal compactification $(\Gamma\bs \D)^\Sigma$ can be understood by completing $\Gamma_J\bs \D$ as follows.  We have the (orbifold) composition
$$\Gamma_J \bs \D  \to \Gamma_J \bs \pi_J(\D) \to (\Gamma_J/\Gamma^J) \bs \pi_{J^\perp}(\D) = (\Gamma_J/\Gamma^J) \bs \H_J,$$
where $(\Gamma_J/\Gamma^J)\subset GL(J)$ is arithmetic.  The first map is a punctured disk bundle, whose filling gives the toroidal compactification.  The Satake-Baily-Borel compactification is obtained by contracting the filled zero section.  Hence, away from the cusps $p_I$, the map $\eps:(\Gamma\bs\D)^\Sigma \to (\Gamma\bs\D)^{SBB}$ is the blow up of the boundary components $\Gamma_J\bs \H_J$.\\\\
Points in the boundary of $(\Gamma\bs\D)^\Sigma$ can be interpreted as mixed Hodge structures of local degenerations, and the morphism $\eps$ forgets the extension data, leaving the associated graded.  Extension data are the same as one-motifs in the sense of Deligne:
\begin{thm} \label{carlson} \cite{carlson}
Let $H^i$ denote a pure integral Hodge structure of weight $i$ (where $i=1,2$).  Extensions in the abelian category MHS of mixed Hodge structures are identified with homomorphisms:
$$\text{Ext}^1_{\text{MHS}}(H^2,H^1) \simeq J^0 \Hom(H^2,H^1),$$
where $J^0 H := H_\C / (F^0 H + H_\Z)$ is the $0$-th intermediate Jacobian.
\end{thm}\noindent
In the relevant case where $H^i = \gr^W_i\Lambda$, $H^2$ is of Tate type, so the Jacobian above simplifies to
$$J^0 \Hom(H^2,H^1)\simeq \Hom_\Z \left(H^2, J^1(H^1)\right) = (H^2)^\vee\otimes_\Z E,$$
which is isomorphic to the abelian variety $E^{16}$.  Geometrically, the group homomorphism corresponds to the restriction of Cartier divisors, followed by summation using the group law on $E$:
$$\gr_2^W \Lambda\simeq H^2_{prim}(R_1) \oplus H^2_{prim}(R_2)  \to E.$$
In conclusion, we now have a concrete description of the extension of the period map (\ref{periodmap}) to $(\Gamma\bs\D)^\Sigma$ in the case of the normal crossing degeneration (\ref{degen}).
\bigskip

\section{Completed Noether-Lefschetz Numbers}
\noindent
In this section, we define topological Noether-Lefschetz numbers for complete one-parameter families $X\to M$ of elliptic K3 surfaces.  Given such a family, we have a period map
$$\rho: M \dasharrow \Gamma\bs \mathscr D,$$
defined away from the singular fibers.  By the valuative criterion of properness, such a map extends to a morphism
$$\ov{\rho}:M \to (\Gamma\bs\mathscr D)^\Sigma.$$
The Noether-Lefschetz divisors $\NL_n$ have Zariski closures
$$\ov{\NL}_n \subset (\Gamma\bs\mathscr D)^\Sigma.$$
The Hodge bundle extends to a line bundle on $(\Gamma\bs \mathscr D)^{SBB}$ by the original construction of \cite{bb}, and then we pull it back via the resolution $\eps: (\Gamma\bs\mathscr D)^\Sigma \to (\Gamma\bs\mathscr D)^{SBB}$.  In a slight abuse of notation, we denote the Chern class of this extension by
$$\lambda \in H^2\left( (\Gamma\bs\mathscr D)^\Sigma,\Q \right).$$
\begin{defn}
The completed Noether-Lefschetz series for $X\to M$ is defined by setting $\alpha = \ov{\rho}_*[M]\in H_2\left( (\Gamma\bs\mathscr D)^\Sigma,\Q \right)$, and then
$$\varphi^\Sigma_\alpha(q) = -\alpha\cdot \lambda + \sum_{n\geq 1} \alpha\cdot \left[\ov{\NL}_n\right]\,q^n.$$
\end{defn}\noindent
The intersection numbers depend on the choice of fan $\Sigma$, but for our application they will not.  The key result of this section applies to families with Type II degenerations in the sense of \cite{kulikov}, but we prove it only for the relevant example (\ref{degen}).
\begin{thm}\label{quasimod}
Let $X\to M$ be a complete family of elliptic K3 surfaces whose singular fibers are normal crossing:  $R_1\cup_E R_2$ with $R_i$ rational elliptic surfaces.  Then the $q$-series $\varphi^\Sigma_\alpha(q)$ is a quasi-modular form of weight 10 for $SL_2(\Z)$, and is independent of $\Sigma$.
\end{thm}
\begin{proof}
By the Hodge theoretic description of the boundary components in Section \ref{compact}, the extension $\ov{\rho}$ meets the boundary of $(\Gamma\bs\D)^{SBB}$ in a finite subset of $\A^1_j$, once for each normal crossing singular fiber $R_1\cup_E R_2$ at the point $j(E)$.  It meets the divisorial boundary of $(\Gamma\bs \D)^\Sigma$ at points corresponding to the restriction map
$$\left[H^2_{prim}(R_1)\oplus H^2_{prim}(R_2) \to E\right] \in O(E_8^{\oplus 2})\bs \Hom_\Z\left((-E_8)^{\oplus 2}, E\right) \simeq W\P^{16}.$$
Using Lemma \ref{split} below, we have a splitting
$$\ov{\rho}_*[M] = \alpha_0 + \alpha_1,$$
where $\alpha_0$ (resp. $\alpha_1$) is supported on the interior (resp. boundary) of $(\Gamma\bs\mathscr D)^\Sigma$.  This allows us to rewrite
$$\varphi^\Sigma_\alpha(q) = \varphi_{\alpha_0}(q) + \varphi^\Sigma_{\alpha_1}(q).$$
Corollary \ref{modularity} tells us that $\varphi_{\alpha_0}(q)$ is a scalar multiple of $E_{10}(q)$, so it suffices to compute $\varphi^\Sigma_{\alpha_1}(q)$.  Since $\alpha_1$ is supported on the fibers of $\eps$, the intersection products take place in weighted projective spaces $W\P^{16}\subset (\Gamma\bs\D)^\Sigma$, which satisfy $H_2(W\P^{16},\Q)\simeq \Q$.  From the local description of the toroidal compactification in Section \ref{compact}, the boundary of the completed Noether-Lefschetz divisor $\ov{\NL}_n$ in the fibers of $\eps$ is given by
$$\partial \ov{\NL}_n \cap  W\P^{16} =  O(E_8^{\oplus 2}) \bs \left(\bigcup_{\substack{v\in (-E_8)^{\oplus 2} \\ (v,v)=-2n}} v^\perp\right) \subset O(E_8^{\oplus 2})\bs \Hom_\Z\left((-E_8)^{\oplus 2}, E\right).$$
The pull-back of the canonical generator $\widetilde{\alpha}\in H_2(W\P^{16},\Q)$ to 
$$\Hom_\Z\left((-E_8)^{\oplus 2}, E\right)\simeq E^{16}$$
is represented by a Weyl group invariant collection of elliptic curves.  A system of linear equations represented by an integer matrix $M$ has $\det(M)^2$ solutions on an elliptic curve.  Hence, $\widetilde{\alpha}\cap v^\perp$ is a quadratic form in $v$.  All Weyl group invariant quadratic forms on a root lattice are scalar multiples of $(v,v)$, so we have
\begin{align*}
\varphi^\Sigma_{\alpha_1}(q) &= c\sum_{v\in E_8^{\oplus 2}} (v,v) \,q^{(v,v)/2}\\
 &= c q \frac{d}{dq}\Theta_{E_8\oplus E_8}(q)\\
 &= cq\frac{d}{dq} E_8(q).
\end{align*}
The last $E_8(q)$ refers to the Eisenstein series.  The Hodge bundle $\lambda$ restricts to a trivial bundle on the fiber $W\P^{16}$, so there is no constant term.
\end{proof}
\noindent
\begin{remark}
By the results of \cite{friedman}, any Type II degeneration of K3 surfaces has a stable model with central fiber
$$R_1 \cup_E R_2,$$
a normal crossing of two rational surfaces along a smooth elliptic curve.  In the case of elliptic K3 surfaces, \cite{brunyate} describes the other possibility as
$$S_t \rightsquigarrow S_0=\F_2 \cup_{E} \Bl_{16}\F_2$$
which corresponds to the boundary component for $J^\perp/J \simeq D_{16}^+$.  We expect quasi-modularity of $\varphi^\Sigma_\alpha(q)$ to hold for such families as well, independent of $\Sigma$.
\end{remark}
\begin{question}
Let $X\to M$ be a family of elliptic K3 surfaces with Type III degenerations.  Is the completed Noether-Lefschetz series $\varphi^\Sigma_{\alpha}(q)$ quasi-modular?  What is the dependence on the choice of fan $\Sigma$?
\end{question}
\begin{lemma}\label{split}
The class $\alpha = \ov{\rho}_*[M]\in H_2\left( (\Gamma\bs\D)^\Sigma,\Q \right)$ can be expressed as
$$\alpha = \alpha_0 + \alpha_1,$$
where $\alpha_0\in H_2( \Gamma\bs\D,\Q)$, and $\alpha_1\in H_2\left(\partial (\Gamma\bs\D)^\Sigma,\Q \right)$, pushed forward by inclusion.
\end{lemma}
\begin{proof}
This follows from the blow up description of $ (\Gamma\bs\D)^\Sigma$, but we a topological proof which will be easier to generalize to compactifications of non-algebraic period spaces.  Consider the Mayer-Vietoris sequence for $U=\Gamma\bs \D$ and $V$ a neighborhood of the fiber $W\P$ of the boundary component over $\A^1_j$:
$$H_2(U)\oplus H_2(V) \to H_2(U\cup V) \to H_1(U\cap V).$$
We have $\alpha\in H_2(U\cup V)$, so the obstruction to splitting lies in $H_1(U\cap V)$, where $U\cap V$ retracts to a circle bundle over $W\P$.  The Gysin sequence computes the latter:
$$0 \to H_2(U\cap V) \to H_2(W\P) \to H_0(W\P) \to H_1(U\cap V) \to 0.$$
The map $H_2(W\P) \to H_0(W\P)$ is an isomorphism since both are rank 1 $\Q$-vector spaces, and the Euler class of the circle bundle is negative on $W\P$.
\end{proof}
\bigskip

\section{Proof of Main Theorem}
\noindent
First, we use the reductions of Section \ref{reduction} to prove that the completed period map $\ov{\rho}:M \to (\Gamma\bs\D)^\Sigma$ meets the Noether-Lefschetz locus transversely.
\begin{lemma}\label{transverse}
If $\deg(L_M)\gg 0$, then the Weierstrass model $X\to B$ can be deformed to one such that
\begin{itemize}
\item The boundary $\ov{\rho}(M)\cap \partial (\Gamma\bs\D)^\Sigma$ is disjoint from $\ov{\NL}_n$.
\item $\rho(M)$ intersects each $\NL_n$ transversely.
\item $\rho(M)$ is disjoint from pairwise intersections of distinct $\NL$ components.
\end{itemize}
\end{lemma}
\begin{proof}
Recall that $\pi:X\to B$ is defined by a global Weierstrass equations with coefficients $a_4\in H^0(L^4)$, $b_6\in H^0(L^6)$, where $L = \omega_B^{-1}\otimes L_M$.  Varying these data gives deformations of $X$.  The restriction of $\pi$ to a divisor $D$ supported on the fibers $B\simeq \Bl_b \P E\to M$ is a Weierstrass model with fundamental line bundle $L|_D$.  We have the restriction long exact sequence
$$H^0(B,L^4) \to H^0(C,L^4|_D) \to H^1(B, L^4\otimes \O(-D))$$
$$H^0(B,L^6) \to H^0(C,L^6|_D) \to H^1(B, L^6\otimes \O(-D)),$$
so any pair of Weierstrass coefficients on $D$ can be lifted to $B$ if and only if 
$$H^1(B,L^4\otimes \O(-D)) = H^1(B,L^6\otimes \O(-D)) = 0.$$
This holds for $\deg(L_M)\gg 0$, by the Kodaira vanishing theorem.\\\\
To prove the first statement of the lemma, take $D$ to be the union of the $b$ singular fibers of $B\to M$, choose the broken K3 surfaces $R_1\cup_E R_2$ over $D$ to be Noether-Lefschetz general, and then lift the coefficients to $B$.  For the remaining transversality statements, we use the fact that the family of immersions $\rho:M \to (\Gamma\bs \D)^\Sigma$ (with varying Weierstrass coefficients) is {\it freely movable} in the sense of \textsection 2 of \cite{fgreer}.  This property follows from surjectivity of the restriction map for $D$ equal to a smooth fiber of $B\to M$.
\end{proof}
\noindent
In light of Lemma \ref{transverse}, we can give a concrete description of the Kontsevich moduli space $\ov{M}_0(X,\ell+nf)$, for $X\to B$ very general.  When $n=0$, we have
$$\ov{M}_0(X,\ell) \simeq \ov{M}_0(B,\ell) \simeq M,$$
by pushing through the zero section $z:B\to X$.  For $n\geq 2$, the moduli space contains a finite reduced set corresponding to smooth rational curves.  All other stable maps are obtained from these examples by adding vertical components which map to singular fibers of $\pi:X\to B$.
\begin{defn}
For $n\geq 2$, let $r_X(n)$ denote the number of smooth rational curves on a very general Weierstrass model $X$ in class $\ell+nf$.
\end{defn}
\noindent
Let $S\to \P^1$ be an elliptic surface.  Its Mordell-Weil group $\MW(S/\P^1)$ is the finitely generated group of sections $S\to \P^1$, or equivalently the $\C(\P^1)$-rational points of the generic fiber.  The short exact sequence of Shioda-Tate compares this group to the N\'{e}ron-Severi group $\NS(S)$:
$$0\to V(S) \to \NS(S) \to \MW(S/\P^1) \to 0.$$
The kernel $V(S)$ is the subgroup spanned by vertical curve classes and the zero section class.  The polarized version of the sequence is defined by taking the orthogonal complement of the sublattice $U=\langle z,f \rangle \subset \NS(S)$:
$$0\to V_{prim}(S) \to \NS_{prim}(S) \to \MW(S/\P^1) \to 0.$$
By Lemma \ref{transverse}, $\NS_{prim}(S)$  has rank $\leq 1$, for the surfaces $S$ in the family $X\to M$.  There are two cases:  
\begin{itemize}
\item $\MW(S/\P^1)\simeq \Z$ so $S$ contains nonzero section curves, or 
\item $V_{prim}(S)\simeq \Z$ so $S$ has an $A_1$ singularity.
\end{itemize}
Indeed, by Brieskorn's simultaneous resolution, the limiting Hodge structure of a surface with an ADE singularity is pure, and it agrees with the Hodge structure of the resolved surface.  Elliptic surfaces $S$ with $A_1$ singularities occur as $\pi^{-1}(\P^1)$, where $\P^1\subset B$ is a line of the ruling tangent to the discriminant curve
$$\Delta = Z(4a_4^3+27b_6^2)\subset B.$$
The minimal resolution a Weierstrass elliptic surface with an $A_1$ singularity has a reducible fiber of Kodaira type $I_2$.
\begin{prop}
If $\MW(S/\P^1)=\Z\sigma$, then $\rho([S])\in \NL_{nr^2}$ for $n=\sigma\cdot z +2$ and $r\in \mathbb N$ arbitrary.  If $S$ has an $A_1$ singularity, then $\rho([S])\in \NL_{r^2}$ for all $r\in \mathbb N$.
\end{prop}
\begin{proof}
If $\sigma\in \NS(S)$ is the class of a nonzero section curve, then its orthogonal projection to $\NS_{prim}(S)$ has self-intersection $-2\sigma\cdot z-4$ by a straightforward lattice calculation.  If $S$ has an $A_1$ singularity, then the exceptional class on the resolved surface has self-intersection $-2$.  In general, we have $\NL_n\subset \NL_{nr^2}$ since $v\in \NS(S)$ implies $rv\in\NS(S)$.
\end{proof}
\begin{prop}
If $\sigma$ is the class of a nonzero section curve on $S = \pi^{-1}(\P^1)$, and $\iota: S\hookrightarrow X$ is the inclusion map, then $\iota_*\sigma = \ell + nf$ with $n = \sigma\cdot z + 2$.
\end{prop}
\begin{proof}
The class $\iota_*\sigma$ can be determined by intersecting with two complementary divisors in $X$:  $z_*[B]$ and $\pi^*$ of any section of $B\to M$.  The shift by 2 occurs because the normal bundle to the section curve is $L^\vee \simeq \O_{\P^1}(-2)$.
\end{proof}
\noindent
This explains the lack of smooth rational curves in class $\ell+f$.  The actual counts $r_X(n)$ can be obtained by intersecting with the Noether-Lefschetz cycles, and then subtracting off the contributions from $A_1$ singularities.
\begin{thm}\label{structure}
For $X$ very general, the counts $r_X(n)$ have the following structure:
\begin{align*}
\sum_{n\geq 2} r_X(n) q^n &= \varphi(q) - \frac{a_1}{2}\Theta_1(q) + \frac{a_1}{2} + \deg_M(\lambda);\\
\Theta_1(q) :&= \sum_{r\in \Z} q^{r^2}.
\end{align*}
Here $\varphi(q)\in \Q[E_2,E_4,E_6]_{10}$, and $a_1$ is the number of $A_1$ singular surfaces.
\end{thm}
\begin{proof}
The completed period map $\ov{\rho}:M\to (\Gamma\bs \D)^\Sigma$ is defined as a map of varieties.  To properly compute the $A_1$ contribution, we lift it to a map of stacks.  Let $M'\to M$ be a double cover ramified at the $a_1$ points where the fiber of $X\to M$ has an $A_1$ singularity.  The total space of the base change $X'\to M'$ contains $a_1$ threefold $A_1$ singularities, and it admits a small resolution $Y$.  The new family $Y \to M'$ is a Brieskorn simultaneous resolution, and it satisfies the assumptions of Theorem \ref{quasimod}, so we have a quasi-modularity statement for the period map $\ov{\rho}':M' \to (\Gamma\bs \D)^\Sigma$:
$$- \ov{\rho}'_* [M']\cdot \lambda +\sum_{n\geq 0} \ov{\rho}'_* [M'] \cap \left[\ov{\NL}_n\right] q^n \in \Q[E_2,E_4,E_6]_{10}.$$
Since the extra sections occur away from the $A_1$ singularities, we have
$$- \ov{\rho}'_* [M']\cdot \lambda +\sum_{n\geq 0} \ov{\rho}'_* [M'] \cap \left[\ov{\NL}_n\right] q^n = - 2\,\ov{\rho}_* [M]\cdot \lambda +2\sum_{n\geq 2} r_X(n) q^n + a_1 \Theta_1(q).$$
The result now follows by adjusting the constant term to 0.
\end{proof}
\noindent
To pass from actual counts to Gromov-Witten invariants, we use the concrete description of $\ov{M}_0(X,\ell+nf)$ and the conifold transition formula of Li-Ruan:
\begin{thm}\cite{aliruan}\label{aliruan}
Suppose that $Y$ and $Y_c$ are Calabi-Yau threefolds related by a conifold transition, that is a small contraction of disjoint $\P^1$'s to $A_1$ singularities, followed by a smoothing deformation.  Then there is a surjective homomorphism $\phi:H_2(Y,\Z) \to H_2(Y_c,\Z)$, and for any homology class $\alpha\in H_2(Y_c,\Z)$,
$$N^{Y_c}_{g,\alpha} = \sum_{\phi(\gamma)=\alpha} N^{Y}_{g,\gamma}.$$
\end{thm}
\noindent
This formula will cancel with the $a_1\Theta_1(q)$ term above.
\begin{thm}\label{gwstructure}
The genus 0 Gromov-Witten invariants of $X$ in the classes $\ell+nf$ have the following structure:
$$F^X_{0,\ell}(q) = \sum_{n\geq 0} N^X_{0,\ell+nf} q^n = \varphi(q)\cdot \frac{q}{\Delta(q)}.$$
\end{thm}
\begin{proof}
Following \cite{mp}, we deform $X'$ by moving the branch points of the double cover $B'\to B$ to general position.  This gives a smoothing $Y_c$ of $X'$, related to $Y$ by a conifold transition.  If we allow the branch points to collide in pairs, $B'$ degenerates to two copies of $B$ glued at $\frac{a_1}{2}$ general points of $B$.  The base change via $\pi:X\to B$ gives a normal crossing degeneration:
$$Y_c \rightsquigarrow X\cup_D X,$$
where $D$ is a disjoint union of smooth K3 surfaces fibers of $X\to M$.  The degeneration formula of \cite{jli} gives $F^{Y_c}_{0,\ell} (q)= 2\, F^X_{0,\ell}(q)$.  Theorem \ref{aliruan} in turn says that
$$N^{Y_c}_{0,\ell+nf} = \sum_{i=1}^{a_1} \sum_{r\in \Z} N^Y_{0,\ell+nf+r\gamma_i},$$
where the $\gamma_i=[\P^1]$ are exceptional curves of the small resolution.  To compute $N^Y_{0,\ell+nf}$, observe that any stable map has a smooth horizontal component of class $\ell+jf$ and a collection of $(n-j)$ vertical components, all inside a K3 surface.  The reduced invariant of the K3 surface is computed in \cite{bryanleung} as $[q/\Delta(q)]_k$, where the self-intersection of the class in the K3 surface is $2k-2$.  This description decomposes the Kontsevich space $\ov{M}_0(X,\ell+nf)$ into closed-open substacks indexed by $j$, so the virtual class integrals sum:
$$N^Y_{0,\ell+nf} = \sum_{j=2}^n 2r_X(j)[q/\Delta(q)]_{n-j} - 2\deg_M(\lambda) [q/\Delta(q)]_n.$$
Stable maps in class $\ell+nf+r\gamma_i$ ($r\neq 0$) are localized to the resolved K3 surface containing $\gamma_i$, where the curve class has self-intersection $2n-2-2r^2$, so for $r\neq 0$,
$$N^Y_{0,\ell+nf+r\gamma_i} = [q/\Delta(q)]_{n-r^2}.$$
Summing up the full potential function, we get non-trivial cancellation:
\begin{align*}
F^{Y_c}_{0,\ell}(q) &= 2\left(  \varphi(q) - \frac{a_1}{2}\Theta_1(q) + \frac{a_1}{2} + \deg_M(\lambda) \right)\cdot\frac{q}{\Delta(q)} - 2\deg_M(\lambda)\cdot\frac{q}{\Delta(q)}\\
 &\,\,\,\,\,\,\,+ a_1 (\Theta_1(q)-1)\cdot \frac{q}{\Delta(q)}\\
 &= 2\,\varphi(q)\cdot\frac{q}{\Delta(q)}.
\end{align*}
\end{proof}
\begin{thm}
If $\varphi(q)$ is the quasi-modular form from Theorems \ref{structure} and \ref{gwstructure}, then
$$\frac{\partial \varphi}{\partial E_2} = -\frac{b}{12} \cdot E_8(q),$$
where $b$ is the number of broken fibers of $B\to M$.
\end{thm}
\begin{proof}
From the proof of Theorem \ref{quasimod}, we know that
$$\varphi(q) = -\deg_M(\lambda)E_{10}(q) + cq\frac{d}{dq} E_8(q).$$
We compute $\deg_M(\lambda)$ by expressing the Hodge bundle as $q_*\left(\omega_{X/M}\right)$, factoring $q$ as $X\os{\pi}\to B\os{f}\to M$, and using the projection formula:
\begin{align*}
q_*\left(\omega_{X/M}\right) &= f_*\pi_*(\omega_{X/B}\otimes \pi^*\omega_{B/M})\\
 &= f_*(\pi_*\omega_{X/B}\otimes \omega_{B/M})\\
 &= f_*(L\otimes \omega_B \otimes f^*\omega_M^{-1})\\
 &= L_M\otimes \omega_M^{-1}.
\end{align*}
Hence $\deg_M(\lambda) = \deg(L_M)+2(1-g)$ for $g=g(M)$.  Next, we compute $a_1$ by counting vertical tangents to the discriminant curve $\Delta\subset B$.  The morphism $\Delta \to M$ has degree 24, and if we compose with the normalization $\nu:\widetilde{\Delta} \to \Delta$, we obtain a morphism of smooth curves.  The number $r$ of ramification points is $a_1+\kappa(\Delta)$, since each cusp contributes once to the ramification.  By the Riemann-Hurwitz formula,
\begin{align*}
a_1 = r - \kappa &= 2g(\Delta)-2 + 48(1-g) - \kappa\\
 &= 2p_a(\Delta)-2 - 3\kappa + 48(1-g) \\
 &= \Delta\cdot(K_B+\Delta) - 3\cdot c_1(L^4)\cdot c_1(L^6) + 48(1-g)\\
 &= 60K_B^2 - 132 K_B\cdot c_1(L_M) + 48(1-g).
\end{align*}
Now since $K_B^2= 8(1-g)-b$ and $K_B\cdot c_1(L_M) = -2\deg(L_M)$, we are left with
$$a_1 = 264 \deg (L_M)+ 528(1-g) - 60b.$$
By Theorem \ref{structure}, we have
$$\sum_{n\geq 2} r_X(n) q^n = \varphi(q) - \frac{a_1}{2}\Theta_1(q) + \frac{a_1}{2} + \deg_M(\lambda).$$
Extracting the coefficient of $q^1$, we find
$$c = \frac{a_1 - 264 \deg_M(\lambda)}{480} = -\frac{b}{8}. $$
The result now follows from this equality and the Ramanujan identity
$$q\frac{d}{dq} E_4 = \frac{E_2E_4 - E_6}{3}.$$\end{proof}

\newpage
\nocite{*}
\bibliography{mixedNL}

\begin{thebibliography}{10}

\bibitem{amrt}
A.~Ash, D.~Mumford, M.~Rapoport, and Y.~Tai.
\newblock {\em Smooth Compactification of Locally Symmetric Varieties},
  volume~4 of {\em Lie Groups: History, Frontiers and Applications}.
\newblock Math Sci Press, 1975.

\bibitem{bb}
W.L. Baily and A.~Borel.
\newblock Compactification of arithmetic quotients of bounded symmetric
  domains.
\newblock {\em Ann. of Math.}, 84(3):442--528, 1966.

\bibitem{beauville}
A.~Beauville.
\newblock {\em Complex Algebraic Surfaces}.
\newblock London Mathematical Society Student Texts. Cambridge University
  Press, 1996.

\bibitem{borch}
R.~Borcherds.
\newblock The {G}ross-{K}ohnen-{Z}agier theorem in higher dimensions.
\newblock {\em Duke J. Math.}, 97(1):219--233, 1999.

\bibitem{brunyate}
Adrian Brunyate.
\newblock {\em A modular compactification of the space of elliptic K3
  surfaces}.
\newblock PhD thesis, University of Georgia, 2015.

\bibitem{bryanleung}
J.~Bryan and N.~Leung.
\newblock The enumerative geometry of {K}3 surfaces and modular forms.
\newblock {\em Journal of the American Mathematical Society}, 13(2):371--410,
  2000.

\bibitem{carlson}
J.~Carlson.
\newblock Extensions of mixed {H}odge structures.
\newblock {\em Journ\'{e}es de g\'{e}ometrie alg\'{e}brique, Angers/France},
  pages 107--127, 1980.

\bibitem{friedman}
R.~Friedman.
\newblock Global smoothings of varieties with normal crossings.
\newblock {\em Annals of Mathematics}, 118(1):75--114, 1983.

\bibitem{fgreer}
F.~Greer.
\newblock Modular forms from {N}oether-{L}efschetz theory.
\newblock math.AG, arXiv:1801.00375.

\bibitem{blowup}
W.~He, J.~Hu, H.-Z. Ke, and X.~Qi.
\newblock Blow-up formulae of high genus {G}romov-{W}itten invariants in
  dimension six.
\newblock math.AG, arXiv:1402.4221.

\bibitem{kmw}
A.~Klemm, K.~Manschot, and T.~Wotschke.
\newblock Quantum geometry of elliptic calabi-yau manifolds.
\newblock {\em Commun. Number Theory Phys.}, 6(4):849--917, 2012.

\bibitem{km}
S.~Kudla and J.~Millson.
\newblock Intersection numbers of cycles on locally symmetric spaces and
  {F}ourier coefficients of holomorphic modular forms in several complex
  variables.
\newblock {\em Publications Math\'{e}matiques de l'IHES}, 71(2):121--172, 1990.

\bibitem{kulikov}
V.~Kulikov.
\newblock Degenerations of {$K3$} surfaces and {E}nriques surfaces.
\newblock {\em Izv. Akad. Nauk SSSR Ser. Mat.}, 41(5):1008--1042, 1199, 1977.

\bibitem{aliruan}
A.-M. Li and Y.~Ruan.
\newblock Symplectic surgery and {G}romov-{W}itten invariants of {C}alabi-{Y}au
  3-folds.
\newblock {\em Inventiones Mathematicae}, 145:151--218, 2001.

\bibitem{jli}
J.~Li.
\newblock A degeneration formula of {GW}-invariants.
\newblock {\em J. Differential Geom.}, 60(2):199--293, 2002.

\bibitem{root}
E.~Looijenga.
\newblock Root systems and elliptic curves.
\newblock {\em Inventiones mathematicae}, 38(1):17--32, 1976.

\bibitem{looij}
E.~Looijenga.
\newblock Compactifications defined by arrangements, {II}: Locally symmetric
  varieties of type {IV}.
\newblock {\em Duke Math. J.}, 119(3):527--588, 2003.

\bibitem{mp}
D.~Maulik and R.~Pandharipande.
\newblock Gromov-{W}itten theory and {N}oether-{L}efschetz theory.
\newblock In {\em A Celebration of Algebraic Geometry}, volume~18 of {\em Clay
  Math. Proc.}, pages 469--506, Providence, RI, 2013. Amer. Math. Soc.

\bibitem{mirandabasic}
R.~Miranda.
\newblock {\em The basic theory of elliptic surfaces}.
\newblock Dottorato di Ricerca in Matematica. ETS Editrice, Pisa, 1989.

\bibitem{obpix}
G.~Oberdieck and A.~Pixton.
\newblock Gromov-{W}itten theory of elliptic fibrations: {J}acobi forms and
  holomorphic anomaly equations.
\newblock math.AG, arXiv:1709.01481.

\end{thebibliography}
\bibliographystyle{plain}

\end{document}